\documentclass[a4paper,leqno,12pt]{amsart}
\usepackage{epsfig}
\usepackage{amsmath}
 \usepackage[pdftex]{hyperref}
\usepackage{amsfonts}
\usepackage{amsmath,amscd}
\usepackage{amssymb}
\usepackage{amsthm}
\usepackage{mathrsfs}
\usepackage[mathscr]{euscript}
\usepackage{graphicx}
\usepackage{xypic}

\theoremstyle{plain}
\newtheorem{theorem}{Theorem}[section]

\newtheorem{lemma}[theorem]{Lemma}
\newtheorem{cor}[theorem]{Corollary}
\newtheorem{prop}[theorem]{Proposition}

\newtheorem{remark}[theorem]{Remark}
\newtheorem{question}[theorem]{Question}
\newtheorem{prob}[theorem]{Problem}

\theoremstyle{definition}
\newtheorem{defn}[theorem]{Definition}

\newcommand{\C}{\mathbb{C}}
\newcommand{\R}{\mathbb{R}}
\newcommand{\Z}{\mathbb{Z}}
\newcommand{\Ker}{\mathit{Ker}}
\newcommand{\colim}{\mathit{colim}}

\numberwithin{equation}{section}
\begin{document}
\title{Inertia groups of high dimensional complex projective spaces}
\vspace{2cm}

\author{Samik Basu, Ramesh Kasilingam}

\email{samik.basu2@gmail.com  ; samik@rkmvu.ac.in  }
\address{Department of Mathematics,
Vivekananda University,
Belur, Howrah - 711202, West Bengal, India.}

\email{rameshkasilingam.iitb@gmail.com  ; mathsramesh1984@gmail.com  }
\address{Statistics and Mathematics Unit,
Indian Statistical Institute,
Bangalore Centre, Bangalore - 560059, Karnataka, India.}

\date{}

\maketitle
\begin{abstract}
For a complex projective space the inertia group, the homotopy inertia group and the concordance inertia group are isomorphic. In complex dimension $4n+1$, these groups are related to computations in stable cohomotopy. Using stable homotopy theory, we make explicit computations to show that the inertia group is non-trivial in many cases. In complex dimension $9$, we deduce some results on geometric structures on homotopy complex projective spaces and complex hyperbolic manifolds.   
\end{abstract}

\section{Introduction}

The study of manifolds in differential topology presents itself through four different classes of equivalence : homotopy equivalence, homeomorphism, PL homeomorphism and diffeomorphism. The classification of manifolds upto these equivalences is a fundamental question in geometry and topology. 

One of the first results in this subject is the result of Milnor (\cite{Mil56}) that there exist smooth manifolds which are homeomorphic to $S^7$ but not diffeomorphic. The above result raises the possibility of non-diffeomorphic smooth structures for a given a topological manifold. We refer to these as inequivalent smoothings. For example, the $7$-sphere has $28$ inequivalent smoothings \footnote{This actually occurs as a consequence of the results in \cite{KM63}.}.

A smooth manifold homeomorphic to $S^m$ is known as a smooth homotopy $m$-sphere\footnote{We know that if $m\geq 5$ as an $m$-manifold homotopy equivalent to $S^m$ is actually homeomorphic to $S^m$.} . The existence of smooth homotopy $m$-spheres was studied in the amazing work of Kervaire and Milnor in \cite{KM63}. The set of diffeomorphism classes of smooth homotopy spheres $\Theta_m$ $(m\geq 5)$ forms a group under the operation of connected sum. It was shown that there exist exotic spheres in a vast majority of dimensions but also that in each dimension there are only finitely many. The proof of the result points to a curious connection to the stable homotopy groups of spheres (denoted by $\pi_n^s$ for $n\geq 0$) whose values are mysterious but are computable using algebraic techniques especially in low dimensions.

A possible way to change smooth structure on a smooth manifold $M^m$, without changing its homeomorphism type, is to take its connected sum $M^m\#\Sigma^m$ with a smooth homotopy sphere $\Sigma^m$. This induces a group action of $\Theta_m$ on the set of smooth structures on the topological manifold $M$. The collection of smooth homotopy spheres $\Sigma^m$  which admit a diffeomorphism $M^m\#\Sigma^m\to M^m$ forms a subgroup $I(M)$ of $\Theta_m$, called the inertia group of $M^m$. 

The calculation of $I(M)$ for an arbitrary manifold $M$ has proven to be a hard problem in general but there are results in certain cases. Tamura (\cite{Tam62}) constructed explicit non-trivial elements in the inertia group for certain $3$-sphere bundles over $S^4$. Examples with non-trivial inertia group are also constructed by Brown and Steer in \cite{BS65}. In every dimension $m$, it was proved that there exists a manifold $M^m$ such that $I(M^m)=\Theta_m$ (\cite{Win75}). In \cite{lev70}, certain non-trivial elements in the intertia group were constructed for many manifolds most notably for simply connected non-spin manifolds in dimensions $8n+2$. 

There are certain cases when the inertia group is $0$. It was proved by Schultz (\cite{Sch71})   that $I(S^p\times S^q)=0$ when $p+q\geq 5$. Kawakubo (\cite{Kaw68}) proved that $I(\mathbb{C}P^m)=0$ for $m\leq 8$. Limitations on the size of the inertia group have been given by Wall (\cite{Wal62}), Browder (\cite{Bro65}), Kosinski (\cite{Kos67}) and Novikov (\cite{Nov65}). There is no systematic approach for computing the inertia groups in general, and many problems are  open. In this paper we are interested in the problem : {\it What are the inertia groups of $\C P^m$ if $m\geq 9$?}

Analogous to the inertia group, for a manifold $M$ one may define the homotopy inertia group $I_h(M)$ and the concordance inertia group $I_c(M)$. $I_h(M)$ (respectively $I_c(M)$) consists of those $\Sigma \in I(M)$  for which the diffeomorphism $M\# \Sigma\cong M$ is homotopic (respectively concordant) to the identity. In \cite{Ram14} it was proved that for a complex projective space all these groups are the same. 

The concordance inertia groups may be understood using homotopy theory. In dimensions $8n+2$, there is a $\Z/2$ summand of $\pi_{8n+2}^s$ generated by $\mu_{8n+2}$ that corresponds to an exotic sphere $\Sigma^{8n+2}\in \Theta_{8n+2}$. In \cite{Far-Jon}, it was proved that $\C P^{4n+1} \# \Sigma^{8n+2}$ is not concordant to $\C P^{4n+1}$ using certain relations in $KO^*$, the cohomology theory induced by real $K$-theory. Hence this also implies that the element $\Sigma^{8n+2}$ does not lie in the inertia group.  

In this paper we try to compute the inertia groups of $\C P^{4n+1}$ using stable cohomotopy. We deduce that the inertia group $I(\C P^9)\cong \Z/2$ or $\Z/4$ (cf. Theorem \ref{inersumm}). This is the first example of a non-trivial inertia group among the inertia groups of complex projective spaces\footnote{The reader may observe that $\C P^{4n+1}$ is a spin manifold.}. We prove that the same phenomena holds also for $\C P^{13}$ by showing $I(\C P^{13}) \supseteq \Z/2$. The techniques involved are explicit calculations using Spanier-Whitehead duality and the knowledge of the stable homotopy groups of spheres in low dimensions at the prime $2$. 

We also make computations at odd primes $p$ making use of the non-triviality of certain elements in the stable homotopy groups proved in \cite{CNL96}. We deduce that under a liberal hypothesis on $n$ the inertia group of $\C P^{4n+1}$ contains non-trivial $p$-torsion. (For the precise statement see Theorem \ref{inersumm}.) It follows that over a large number of dimensions the inertia groups of complex projective spaces are non-trivial.

Computations of the inertia groups carry with them a number of geometric applications. Using previously known results on the triviality of the inertia group $I(\C P^m)$ for $m\leq 8$,  it is possible to classify, up to diffeomorphism, all closed smooth manifolds homeomorphic to the complex projective $n$-space where $n=3$ and $4$ (\cite{Ram15}). In \cite{Ram14}, it was shown that a way to generate exotic spheres which are not in the inertia group\footnote{These are defined to be Farrell-Jones spheres.} in dimensions $8n+2$ is by considering exotic spheres with a non-trivial $\alpha$-invariant\footnote{These are defined as Hitchin spheres.}. The exotic spheres outside the inertia group for complex projective spaces have consequences for complex hyperbolic manifolds by \cite{Far-Jon}. 

In this paper, we show that the inertia group of $\C P^9$ is a proper subgroup of $\Z/8\subset \Theta_{18}=\Z/2\oplus \Z/8$.   As an application, we give examples of two inequivalent smooth structures of $\mathbb{C}P^{9}$ such that one admits a metric of nonnegative scalar curvature and the other does not (cf. Theorem \ref{exocom}). Following this example, we construct examples of closed negatively curved Riemannian $18$-manifolds, which are homeomorphic but not diffeomorphic to complex hyperbolic manifolds (cf. Theorem \ref{exocomhyp}).\\

\noindent
{\bf Organisation of the paper:} In section 2, we introduce some preliminaries on the inertia group $I(\mathbb{C}P^m)$ and prove a result relating this to a computation in stable cohomotopy for dimension $m=4n+1$. In section 3, we make some computations in stable homotopy related to the above question and prove the non-triviality of $I(\mathbb{C}P^{4n+1})$ for the different values of $n$. Finally, section 4 contains some geometric applications.\\

\noindent
{\bf Notation: } Denote by $O=\colim_n O(n)$, $Top= \colim_n Top(n)$, $G=\colim_n G(n)$ the direct limit of the groups of orthogonal transformations, homeomorphisms, and homotopy equivalences respectively. In this paper all manifolds will be closed smooth, oriented and connected, and all homeomorphisms and diffeomorphisms are assumed to preserve orientation, unless otherwise stated.

\section{Inertia groups of Complex Projective Spaces}
In this section we recall some basic facts about inertia groups specializing to the case of complex projective spaces and provide the background of the arguments in the rest of the paper. We start by recalling some terminology from \cite{KM63}:
\begin{defn}
(a) A homotopy $m$-sphere $\Sigma^m$ is an oriented smooth closed manifold homotopy equivalent to the standard unit sphere $S^m$ in $\mathbb{R}^{m+1}$.\\
(b) A homotopy $m$-sphere $\Sigma^m$ is said to be exotic if it is not diffeomorphic to $S^m$.\\
(c) Two homotopy $m$-spheres $\Sigma^{m}_{1}$ and $\Sigma^{m}_{2}$ are said to be equivalent if there exists an orientation preserving diffeomorphism $f:\Sigma^{m}_{1}\to \Sigma^{m}_{2}$.
\end{defn}

The set of equivalence classes of homotopy $m$-spheres is denoted by $\Theta_m$. The equivalence class of $\Sigma^m$ is denoted by [$\Sigma^m$]. When $m\geq 5$, $\Theta_m$ forms an abelian group with group operation given by connected sum $\#$ and the zero element represented by the equivalence class  of $S^m$. M. Kervaire and J. Milnor \cite{KM63} showed that each  $\Theta_m$ is a finite group.

\begin{defn}
Let $M$ be a topological manifold. Let $(N,f)$ be a pair consisting of a smooth manifold $N$ together with a homeomorphism $f:N\to M$. Two such pairs $(N_{1},f_{1})$ and $(N_{2},f_{2})$ are concordant provided there exists a diffeomorphism $g:N_{1}\to N_{2}$ such that the composition $f_{2}\circ g$ is topologically concordant to $f_{1}$, i.e., there exists a homeomorphism $F: N_{1}\times [0,1]\to M\times [0,1]$ such that $F_{|N_{1}\times 0}=f_{1}$ and $F_{|N_{1}\times 1}=f_{2}\circ g$. The set of all such concordance classes is denoted by $\mathcal{C}(M)$.
\end{defn}

 Note that there is a homeomorphism $h: M^m\#\Sigma^m \to M^m$ $(m\geq5)$ which is the inclusion map outside of homotopy sphere $\Sigma^m$ and well defined up to topological concordance. We will denote the class of $(M^m\#\Sigma^m, h)$  in $\mathcal{C}(M)$  by $[M^m\#\Sigma^m]$. (Note that $[M^m\#S^m]$ is the class of $(M^m, id_{M^m})$.)\\

\begin{defn}
Let $M^m$ be a closed smooth $m$-dimensional manifold. The inertia group $I(M)\subset \Theta_{m}$ is defined as the set of $\Sigma \in \Theta_{m}$ for which there exists a diffeomorphism $\phi :M\to M\#\Sigma$.\\
The homotopy inertia group $I_h(M)$  is defined as the set of all $\Sigma\in I(M)$ such that there exists a diffeomorphism $M\to M\#\Sigma$ which is homotopic to $\rm{id}:M \to M\#\Sigma$.\\
The concordance inertia group $I_c(M)$ is defined as the set of all $\Sigma\in I_h(M)$ such that $M\#\Sigma$ is concordant to $M$.
\end{defn}

We recall the following theorem about complex projective spaces
\begin{theorem}{\rm \cite[Theorem 4.2]{Ram14}}\label{first}
For $n \geq 1$, $I_c(\mathbb{C}P^{n})=I_h(\mathbb{C}P^{n})=I(\mathbb{C}P^{n}).$
\end{theorem}

Next we recall a reformulation of inertia groups via homotopy theory. Let $f_{M}:M^m\to S^m $  be a degree one map. Note that $f_{M}$ is well-defined up to homotopy. Composition with $f_{M}$ defines a homomorphism
$$f_{M}^*:[S^m, Top/O]\to [M^m ,Top/O],$$ 
and in terms of the identifications
\begin{center}
$\Theta_m=[S^m, Top/O]$ and $\mathcal{C}(M^m)=[M^m ,Top/O]$
\end{center}
given by \cite[p. 25 and 194]{KS77}, $f_{M}^*$ becomes $[\Sigma^m]\mapsto [M^m\#\Sigma^m]$. Therefore the concordance inertia group $I_c(M)$ can be identified with $\Ker(f_M^\ast)$.

Recall that the based homotopy classes $[X,G]$ can be identified with the $0^{th}$ stable cohomotopy group $\tilde{\pi}^0(X)$. We also write $f_{M}^*$ for the induced homomorphism $[S^m,G]=\tilde{\pi}^0(S^m)=\pi_m^s\to [M^m,G]=\tilde{\pi}^0(M^m)$ by a degree one map $f_{M}:M^m\to S^m$.

Now recollect some facts from smoothing theory \cite{Bru71a}. The natural  inclusions of $H$-spaces $O\subset Top \subset G$ induce $H$-space maps $\phi:G\to G/O$,  $\psi: Top/O\to G/O$ such that
$$\psi_{*}:\Theta_{8n+2}=\pi_{8n+2}(Top/O)\to \pi_{8n+2}(G/O)$$
is an isomorphism for $n\geq 1$. The homotopy groups of $G$ are the stable homotopy groups of spheres $\pi^{s}_m$ ; i.e., $\pi_{m}(G)=\pi^{s}_m$ for $m\geq 1$. For $n\geq 1$,
$$ \phi_{*}:\pi^{s}_{8n+2}\to \pi_{8n+2}(G/O)$$
is an isomorphism.

\begin{theorem}\label{coniner}
Let $M^{8n+2}$ be a closed smooth $8n+2$-manifold homotopy equivalent to $\mathbb{C}P^{4n+1}$ $(n\geq 1)$. Then $I_c(M)={\rm Ker}(\Phi)$, where $$\Phi: \Theta_{8n+2}\stackrel{\psi_{*}}{\rightarrow}\pi_{8n+2}(G/O)\stackrel{\phi_{*}^{-1}}{\rightarrow}\tilde{\pi}^0(S^{8n+2})\stackrel{f_{M}^*}{\longrightarrow}\tilde{\pi}^0(M^{8n+2}).$$
\end{theorem}

The proof of Theorem \ref{coniner} requires two facts we prove below
\begin{lemma}\label{lem1}
Let $M^{2m}$ be a closed smooth $2m$-manifold homotopy equivalent to $\mathbb{C}P^{m}$ $(m\geq 1)$. Then the homomorphism $\psi_{*}:[M^{2m},Top/O]\mapsto [M^{2m}, G/O]$ is monic.
\end{lemma}
\begin{proof}
Consider the Barratt-Puppe sequence for the inclusion $i:\mathbb{C} P^{m-1}\hookrightarrow \mathbb{C} P^{m}$ which induces the exact sequence on taking homotopy classes $[-,\Omega(G/Top)]$
$$\cdots \to [S\mathbb{C} P^{m-1}, \Omega(G/Top)]{\to}[S^{2m},\Omega(G/Top)] \stackrel{f^{*}_{\mathbb{C} P^{m}}}{\to} [\mathbb{C} P^{m}, \Omega(G/Top)] \stackrel{i^{*}}{\to}[\mathbb{C} P^{m-1}, \Omega(G/Top)]\cdots $$
and by identifying 
$$[S^{2m}, \Omega(G/Top)]=[S^{2m+1}, G/Top]=L_{2m+1}(e)=0,$$ 
where $L_{k}(e)$ is the simply connected surgery obstruction group, and $[\mathbb{C} P^{1}, \Omega(G/Top)]=0$, we can prove that $[\mathbb{C}P^{m}, \Omega(G/Top)]=0$ $(\forall~ m)$. Now consider the following long exact sequence associated to the fibration $Top/O\to G/O\to G/Top:$ 
$$\cdots \to [M^{2m},\Omega(G/Top)]\to [M^{2m},Top/O]\stackrel{\psi_{*}}\to [M^{2m},G/O] \to [M^{2m},G/Top] $$ 
and using the fact that $[M^{2m},\Omega(G/Top)]=[\mathbb{C}P^m,\Omega(G/Top)]=0$, we have that the homomorphism $\psi_{*}$ is monic.
\end{proof}
\begin{lemma}\label{lem2}
Let $M^{2m}$ be a closed smooth $2m$-manifold homotopy equivalent to $\mathbb{C}P^{m}$ $(m\geq 1)$. Then the homomorphism $\phi_{*}:[M^{2m},G]\to [M^{2m},G/O]$ is monic.
\end{lemma}
\begin{proof}
Since $M^{2m}$ is homotopy equivalent to $\mathbb{C}P^{m}$, let $g:M^{2m} \to \mathbb{C}P^{m}$ be a homotopy equivalence. Therefore the induced map $g^{*}:[\mathbb{C}P^m,-]\to [M^{2m},-]$ fits into the following commutative diagram:
$$
\begin{CD}
[\mathbb{C}P^{m},G]@>\phi_{*}>>[\mathbb{C}P^{m},G/O]\\
@V\cong Vg^{*}V      @V\cong Vg^{*}V\\
[M^{2m},G]  @>\phi_{*}>> [M^{2m},G/O]
\end{CD}
$$
Since Brumfiel \cite[p.77]{Bru71} has shown that $$\phi_{*}:[\mathbb{C}P^{m},G]\to [\mathbb{C}P^{m},G/O]$$ is monic for all 
$m\geq 1$. This implies that the homomorphism $\phi_{*}:[M^{2m},G]\to [M^{2m},G/O]$ is monic. 
\end{proof}
\paragraph{\it Proof of Theorem \ref{coniner}:}
Consider the following commutative of diagram :
$$
\begin{CD}
[S^{2m},Top/O]=\Theta_{2m}@>f^*_{M^{2m}}>> [M^{2m},Top/O]=\mathcal{C}(M^{2m})\\
@VV\psi_{*}V                      @VV\psi_{*}V\\
[S^{2m},G/O]  @>f^*_{M^{2m}}>> [M^{2m},G/O]\\
@AA\phi_{*}A                      @AA\phi_{*}A\\
[S^{2m},G]=\pi_{2m}^{s} @>f^*_{M^{2m}}>>  [M^{2m},G]
\end{CD}
$$
Recall that the concordance class $[M^{2m}\#\Sigma]\in [M^{2m},Top/O]$ of $M^{2m}\#\Sigma$ is $f^{*}_{M^{2m}}([\Sigma])$ when $m > 2$, and that $[M^{2m}]=[M^{2m}\#S^{2m}]$ is the zero element of this group. Now Lemma \ref{lem1} and Lemma \ref{lem2}  used in conjunction with a simple diagram chase for $m=4n+1$ show that $I_c(M^{8n+2})={\rm Ker}(\Phi=f_{M}^*\circ \phi_{*}^{-1}\circ \psi_{*})$ thus proving Theorem \ref{coniner}. \qed

Identifying the group $\Theta_{8n+2}$ with $\tilde{\pi}^0(S^{8n+2})$ in view of Theorem \ref{coniner} we consider the following question related to inertia groups.
\begin{prob}\label{main}
What is the kernel of $f_{\C P^{4n+1}}^\ast : \tilde{\pi}^0(S^{8n+2}) \to \tilde{\pi}^0(\C P^{4n+1})$ ?
\end{prob}
We explore some cases of this question in the next section.

\section{Computations in stable homotopy}
In this section we make computations relating to Problem \ref{main}. We work in the category of spectra. Throughout this section we use the notation $X$ for both the space and the spectrum $\Sigma^\infty X$ and the notation $\{X,Y\}$ for the stable homotopy classes of maps from $X$ to $Y$. 

Recall that (\cite{EKMM}, \cite{sym}) there are  models of the category of spectra which are a closed symmetric monoidal category with monoidal structure given by the smash product $\wedge$ and the mapping spectrum denoted by $F(-,-)$. The sphere spectrum $S^0$ is the unit of the monoidal structure. In this category for a spectrum $X$ one may form the dual spectrum $DX = F(X,S^0)$. This notion appeared earlier for finite spectra as the Spanier-Whitehead dual of $X$ (\cite{Ada1}). 

One notes that if $X$ is a finite cellular spectrum then so is the dual $DX$. We briefly recall this construction. Let the cellular structure on $X$ be given by $X=\mathit{colim}_n X^{(n)}$ such that $X^{(n)}$ is obtained from $X^{(n-1)}$ by attaching a cell of dimension $a_n$. That is, there are cofibre sequences 
$$S^{a_n-1} \to X^{(n-1)} \to X^{(n)}.$$ 
The dual structure on $DX$ is given by $DX= \mathit{colim}_n D(X/X^{(n)})$ with cofibre sequences $$S^{-a_n-1} \to D(X/X^{(n)}) \to D(X/X^{(n-1)}) $$
obtained by dualizing the cofibre 
$$S^{a_n}\to X/X^{(n-1)}\to X/X^{(n)} \to S^{a_n+1}$$
Thus, in the colimit   $DX= \mathit{colim}_n D(X/X^{(n)})$, $D(X/X^{(n-1)})$ is obtained from $D(X/X^{(n)})$ by attaching a cell of dimension $-a_n$. Therefore, $DX$ is also a finite cellular spectrum with a $-n$-cells for every $n$-cell of $X$. 

Note that $\tilde{\pi}^0(X) \cong \{X,S^0\}$. Note also that $\{X,S^0\} = \pi_0 F(X,S^0) = \pi_0 D(X)$. Therefore for a map $f: X\to Y$ of spectra the map $\tilde{\pi}^0(f): \tilde{\pi}^0 (Y) \to \tilde{\pi}^0(X)$ is equivalent to the map $\pi_0(D(f)): \pi_0D(Y) \to \pi_0 D(X)$ induced by the dual map $D(f): D(Y) \to D(X)$. 

For Problem \ref{main} we wish to compute $\tilde{\pi}^0(f)$ where $f: \C P^{4n+1} \to S^{8n+2}$ is the usual degree $1$ map. Our approach is to compute $\pi_0(D(f))$. 

\subsection{Computations in dimension 18}\label{comp18}
We begin by noting that $D\C P^9$ has the filtration 
\begin{equation}\label{filt}
S^{-18}= D(\C P^9/\C P^8) \to D(\C P^9/\C P^7) \cdots D(\C P^9/\C P^0)= D(\C P^9)
\end{equation}
and there are cofibre sequences for $1\leq k \leq 8$
$$S^{-2k-1} \to D(\C P^9/\C P ^{k}) \to D(\C P^9/ \C P^{k-1}) \to S^{-2k}$$
The degree $1$ map $\C P^9 \to S^{18}$ dualises to the inclusion of the bottom cell $S^{-18} \to D(\C P^9)$. Thus we are interested in the question : which elements of $\pi_{18}^s = \pi_0 S^{-18}$ map to $0$ under the map above? Recall that the homotopy group $\pi_0 S^{-18} = \pi_{18}^s=\Z/2\oplus \Z/8$. 
The $\Z/2$ summand is generated by the element $\mu_{18}$ (\cite{Ada}). From \cite{Far-Jon} it follows that this element maps non trivially into $\pi_0 D\C P^9$. Therefore the question remains about the other summand $\Z/8$. In this computation we need formulas for the action of the Steenrod operations on the cohomology of $\C P^n$. We recall them below
\begin{prop} \label{St}
In $H^*(\C P^n;\Z/2)$ we have the formula 
$$Sq^2(x^k)=x^{k+1} \iff  k \mbox{ is odd}.$$
$$Sq^4(x^k)=x^{k+2} \iff k \equiv 2~\mathit{or}~3~ (\mathit{mod}~4)$$
\end{prop}

Recall that the Steenrod operation $Sq^2$ detects the Hopf map $\eta$ which in our notation is $h_1$ and $Sq^4$ detects the map $\nu$ which in our notation is $h_2$ (modulo $2h_2$). We prove the Proposition
\begin{prop}\label{inj18}
The map $\pi_{18}^s = \pi_0 S^{-18} \to \pi_0 (D (\C P^9/ \C P^3))$ is injective. 
\end{prop}

\begin{proof}
We compute this map using the filtration \ref{filt}. Note that the group $\pi_{18}^s$ is $2$-torsion so it suffices to work in the $2$-local category. We use the notation in \cite{Rav}, Appendix 3.3. In terms of this notation $\pi_{18}^s$ is $\Z_2\{h_1P^2h_1\} \oplus \Z_8\{h_2h_4\}$. It helps to note that the element $h_2h_4$ is indecomposable in the algebra $\pi_*^s$ of stable homotopy groups (since the element $h_4$ supports the differential $d_2(h_4)=h_0 h_3^2$ and all the hidden extensions in the range are written in Cor 4.4.50). \\

We start proceeding along the sequence \ref{filt} with the spectrum $D(\C P^9/ \C P^7)$. This fits into a cofibre 
$$S^{-17} \to S^{-18} \to D(\C P^9/\C P^7)  $$
The map $S^{-17} \to S^{-18} \in \pi_1^s = \Z/2\{h_1\}$. Whether it is non-trivial or not is determined by the action of the Steenrod operation $Sq^2$ on the cone and hence determined by the action of $Sq^2$ on $\C P^9/\C P^7$. Note from Proposition \ref{St} that $Sq^2(x^8)=0$. Therefore the map is trivial and $D(\C P^9/\C P^7) \simeq S^{-18}\vee S^{-16}$. It follows that the map from $\pi_{18}^s \to \pi_0 D(\C P^9/ \C P^7)$ is injective. \\

The next term is $D(\C P^9/ \C P^6)$. We have the cofibre sequence 
\begin{equation}\label{cof6}
S^{-15} \to D(\C P^9/ \C P^7) \to D(\C P^9/\C P^6) 
\end{equation}
The map $S^{-15} \to D(\C P^9/\C P^7) \simeq S^{-18}\vee S^{-16}$ is an element of $\pi_1^s \oplus \pi_3^s = \Z/2\{h_1\}\oplus \Z/2\{h_2\}$. Note that on $\C P^9/\C P^6$ the Steenrod operations satisfy the formulas $Sq^2(x^7)=x^8$ and $Sq^4(x^7)=x^9$. Thus the $16$-cell in $\C P^9/\C P^6$ attaches onto the $14$-cell by $h_1$ and the $18$-cell attaches via $h_2$ onto the $14$-cell (or some other odd multiple which does not change the computations below). Therefore the map $S^{-14} \to S^{-16} \vee S^{-18}$  is given by $(h_1,h_2)$. On $\pi_0$ we have the sequence 
$$\pi_{15}^s \stackrel{(h_1,h_2)}{ \to} \pi_{16}^s \oplus \pi_{18}^s \to \pi_0 D(\C P^9/\C P^6)$$ 
Observe that multiplication by $h_2$ from $\pi_{15}^s$ to $\pi_{18}^s$ is $0$.  It follows that the map from $\pi_{18}^s \to \pi_0 D(\C P^9/ \C P^6)$ is injective. \\

The next term in the sequence \ref{filt} is $D(\C P^9/ \C P^5)$ and it is formed from $D(\C P^9/ \C P^6)$ by the cofibre
$$S^{-13} \to D(\C P^9/ \C P^6) \to D(\C P^9/\C P^5)  $$
The group $\pi_{13}^s = 0$ and so the map from $\pi_{18}^s \to \pi_0 D(\C P^9/ \C P^5)$ is injective. \\

Next we analyse the cofibre 
$$S^{-11} \to D(\C P^9/ \C P^5) \to D(\C P^9/\C P^4)  $$
and compute the image $\pi_{11}^s = \Z/8\{Ph_2\} \to \pi_0 D(\C P^9/ \C P^5)$. Note that the methods above imply that $D(\C P^9/ \C P^5)$ is the cofibre 
\begin{equation} \label{cof5}
S^{-13}\vee S^{-15} \stackrel{A}{\to} S^{-16}\vee S^{-18} \to  D(\C P^9/ \C P^5)
\end{equation}
where the matrix $A$ is given by 
$\begin{bmatrix}
    h_2 & h_1  \\
    0 & h_2
\end{bmatrix}$
. The map $S^{-11} \to D(\C P^9/ \C P^5)$ may be described by a map $S^{-11} \to S^{-12} \vee S^{-14}$ (hence $\in \pi_1^s \oplus \pi_3^s =\Z/2\{h_1\}\oplus \Z/8\{h_2\}$) together with a choice of null homotopy after composition by $A$ to $S^{-15}\vee S^{-17}$. The choice of null homotopy lies in $\{S^{-11}, S^{-16}\vee S^{-18}\} = \pi_5^s\oplus \pi_7^s= \Z/16 \{h_3\}$. 

The map $S^{-11} \to S^{-12}\vee S^{-14}$ is also the attaching map of the $(-10)$-cell for the spectrum $D(\C P^7/\C P^4)$. Hence one may try to compute its homotopy class via Steenrod operations. From the formulas in Proposition \ref{St}  we have $Sq^2(x^5)= x^6$ and $Sq^4(x^5)=0$. So the map onto the $(-12)$-cell is $h_1$ and the map onto the $(-14)$-cell is some even multiple of $h_2$. It follows from the formulas in (\cite{Mos68}, Proposition 5.2) that the second map is $0$. 

Now take a class $a \in \pi_{11}^s$, so that $a=$ some multiple of $Ph_2$. To compute its image onto $\pi_{12}^s \oplus \pi_{14}^s$ one has to multiply with the class $(h_1,0)$. Notice that $h_1Ph_2=0$. Therefore the image must be zero. 

Our case of interest is the image in $\pi_{18}^s$. This can be computed using Toda brackets. The image onto $\pi_{18}^s$ must be obtained by map onto the factor $S^{-14}$ which is the only factor which attaches non-trivially down to the $(-18)$-cell. However the map from $S^{-11}$ is null-homotopic on the $(-14)$-cell so this map does not hit any element of $\pi_{18}^s$. In addition the indeterminacy for the bracket is trivial. Hence the map from $\pi_{18}^s \to \pi_0 D(\C P^9/ \C P^4)$ is injective. \\

The next term in the sequence \ref{filt} is $D(\C P^9/ \C P^3)$. This fits into a cofibre 
$$S^{-9} \to D(\C P^9/ \C P^4) \to D(\C P^9/\C P^3)  $$
and we compute the image $\pi_9^s = \Z/2\{h_2^3=h_1^2h_3,h_1c_0,Ph_1\} \to \pi_0 D(\C P^9/ \C P^4)$. Compose the map $S^{-9} \to D(\C P^9/\C P^4) \to S^{-10}$ to the top cell. This can be detected by computing $Sq^2$. As $Sq^2(x^4)=0$ this map is null-homotopic. Therefore the attaching map goes down to $D(\C P^9/\C P^5)$ which we compute using the cofibre \ref{cof5}.

The map $S^{-9} \to D(\C P^9/ \C P^5)$ is given by a map $S^{-9} \to S^{-12} \vee S^{-14}$ (hence $\in \pi_3^s \oplus \pi_5^s = \Z/8\{h_2\}$) and a choice of null homotopy after composition by $A$ to $S^{-15}\vee S^{-17}$. The choice of null homotopy lies in $\{S^{-9}, S^{-16}\vee S^{-18}\} = \pi_7^s\oplus \pi_9^s= \Z/16 \{h_3\}\oplus \Z/2\{h_2^3,h_1c_0,Ph_1\}$.  We have $Sq^2(x^4)= 0$ and hence the map is of the form $2kh_2$. From the formulas in (\cite{Mos68}, Proposition 5.2) it follows that $k=1$. Now take a class $a \in \pi_{9}^s$. Note $h_2\cdot \pi_9^s =0$.  Thus the image in $\pi_{12}^s$ is $0$. 

Our interest is the image in $\pi_{18}^s$. This can again be computed using Toda brackets. The image onto $\pi_{18}^s$ must be obtained by map onto the factor $S^{-14}$ which is the only factor which attaches down to the $(-18)$-cell. But the attachment of $S^{-9}$ onto this cell is $0$ and hence the entire Toda bracket is forced to contain zero. Also the indeterminacy may be computed to be $0$. Hence the map from $\pi_{18}^s \to \pi_0 D(\C P^9/ \C P^3)$ is injective. 
\end{proof}

The above computation implies that the classes in $\pi_{18}^s$ survive in the sequence all the way upto $\pi_0 D(\C P^9/\C P^3)$. However in the next step we obtain a non-trivial kernel. We use some formulas for Toda brackets from \cite{Koch96}.
\begin{theorem}\label{main18}
The class $4h_2h_4\in \pi_{18}^s$ maps to $0$ in $\pi_0(D(\C P^9/\C P^2))$. It follows that the kernel of the map $\pi_{18}^s \to \pi_0(D\C P^9)$ is at least $\Z/2$. 
\end{theorem}

\begin{proof}
The second statement follows from the first and the fact that $h_2h_4$ represents an element of order $8$ in $\pi_{18}^s$. Thus we need prove only the first. We have the cofibre 
$$S^{-7} \to D(\C P^9/ \C P^3) \to D(\C P^9/\C P^2)  $$
We compute the image $\pi_7^s = \Z/16\{h_3\} \to \pi_0 D(\C P^9/ \C P^3)$. Compose the map $S^{-7} \to D(\C P^9/\C P^3) \to S^{-8}$ to the top cell. This composite is $h_1$ by the formula  $Sq^2(x^3)=x^4$ in Proposition \ref{St}. The map $h_1 : \pi_7^s \to \pi_8^s$ has kernel $\Z/8\{2h_3\}$. 

First we prove that the map $\pi_7^s \to \pi_0D(\C P^7/\C P^3)$ has kernel $\Z/8\{2h_3\}$. That is, we show that $2h_3$ maps to $0$ in the latter group. Note from the proof of Proposition \ref{inj18} that none of the cells in dimension $-8$, $-10$, $-12$ attach to the $(-14)$-cell. Thus we have $D(\C P^7/ \C P^3) \simeq D(\C P^6/\C P^3) \vee S^{-14}$. 

Observe that the $(-8)$-cell of $D(\C P^7/ \C P^3)$ does not attach to the $(-10)$-cell. Together with the fact that $\pi_{12}^s$ and $\pi_{13}^s$ are $0$ we get that $\pi_0D(\C P^6/\C P^3)\to \pi_0(S^{-8}\vee S^{-10})$  is an isomorphism. Now observe that  the composite map $S^{-7}\to S^{-8}\vee S^{-10}$ is $h_1$ on the first factor and $h_2$ on the second factor. Hence the kernel of the map $\pi_7^s \to \pi_0D(\C P^6/\C P^3)$ is $\Z/8\{2h_3\}$. 

Observe that Proposition 5.6 of \cite{Mos68} implies that the $14$-cell in $\C P^7/\C P^2$ does not attach to the cells in dimension $8,10,12$ and attaches onto the $6$-cell by the map $2h_3$. Hence the map $S^{-7} \to D(\C P^6/\C P^3)\vee S^{-14} \to S^{-14}$ is given by $2h_3$. Thus the map $\pi_7^s \to \pi_{14}^s$ is multiplication by $2h_3$ which is $0$. Thus we have that the Kernel of the map $\pi_7^s \to \pi_0D(\C P^7/\C P^3)$ is $\Z/8\{2h_3\}$.

Next we extend the above result to the map $\pi_7^s \to \pi_0 D(\C P^8/ \C P^3)$. The map $2h_3 : S^0 \to S^{-7}$ factors in the diagram as below
$$\xymatrix{S^0 \ar[r]^{2h_3} \ar[rd]_{h_3}   & S^{-7} \ar[rr] & & D(\C P^8/\C P^3) \ar[r] & S^{-8} \\ 
                         &   S^{-7} \ar[u]_{2} \ar[rr]^{\alpha}        & & D(\C P^8/\C P^4) \ar[u] }$$
We prove that the composite 
$$ S^0 \stackrel{h_3}{\to} S^{-7} \stackrel{\alpha}{\to} D(\C P^8/\C P^4)$$ 
is $0$. The above argument shows that 
$$ S^0 \stackrel{h_3}{\to} S^{-7} \stackrel{\alpha}{\to} D(\C P^8/\C P^4)\to D(\C P^7/ \C P^4)\simeq D(\C P^6/\C P^4) \vee S^{-14}$$ 
is 0. It remains to compute the map onto $\pi_{16}^s = \pi_0 S^{-16}$ the bottom cell of $D(\C P^8/\C P^4)$. This may be computed via the cofibre 
$$ S^{-16} \to D(\C P^8/\C P^4)\to D(\C P^7/ \C P^4) \to S^{-17}$$
as an element of the sum of Toda brackets $\langle h_3, \alpha, D(\C P^6/\C P^4) \to S^{-15} \rangle$ and $\langle h_3, \alpha, S^{-14} \to S^{-15}\rangle$. 

We use that the attachment of the $(-6)$-cell to the $(-14)$-cell is $2h_3$ as noted above. Thus the attachment of $\alpha$ onto the $(-14)$-cell is given by  $4h_3$. Hence the latter bracket equals $\langle h_3, 2h_3, h_1\rangle = 2h_1h_4=0$ modulo trivial indeterminacy. 

The first bracket above is the $4$-fold bracket $\langle h_3, 2h_2, h_1,h_2\rangle$. The indeterminacy of this bracket lies in the three fold Toda bracket $\langle Ph_2, h_1,h_2\rangle$. Note that this bracket is in the kernel of the map $\pi_{16}^s \to \pi_0 D(\C P^8/\C P^4)$ killed by the attachment of the $(-10)$-cell. Modulo the above indeterminacy the bracket $\langle h_3, 2h_2, h_1,h_2\rangle$ is a multiple of $2$ and hence $0$. Thus, the kernel of $\pi_7^s \to \pi_0 D(\C P^8/ \C P^3)$ is $2h_3$. 

The class $2h_3$ maps to $0$ under $S^{-7} \to D(\C P^9/\C P^3) \to D(\C P^8/\C P^3)$. Thus it maps to $\pi_0S^{-18}$. We prove that it maps to the class $4h_2h_4$ under this map. We have the factorization as above
$$\xymatrix{S^0 \ar[r]^{2h_3} \ar[rd]_{h_3}   & S^{-7} \ar[rr] & & D(\C P^9/\C P^3) \ar[r] & S^{-8} \\ 
                         &   S^{-7} \ar[u]_{2} \ar[rr]^{\alpha}        & & D(\C P^9/\C P^4) \ar[u] }$$ 
We write $D(\C P^9/\C P^4)$ as the cofibre 
$$\Sigma^{-1} D(\C P^6/ \C P^4) \vee S^{-15} \to S^{-16}\vee S^{-18} \to D(\C P^9/\C P^4) \to D(\C P^6/ \C P^4) \vee S^{-14}\to  S^{-15}\vee S^{-17}$$
Hence the map onto $\pi_{18}^s$ is a sum of Toda brackets $\langle h_3, \alpha, D(\C P^6/ \C P^4) \to S^{-17}\rangle$ and $\langle h_3, \alpha, S^{-14} \to S^{-17}\rangle$. The latter map is the Toda bracket $\langle h_3, 4h_3,h_2\rangle = 2h_2h_4$  modulo trivial indeterminacy. 

For the first map note that the attaching map $D(\C P^6/\C P^4) \to S^{-17}$ restricting to the bottom cell $S^{-12}$ is trivial and on the top cell is $h_3$ as computed by the Steenrod operation $Sq^8(x^5)=x^9$. Therefore the first map is $\langle h_3, 2h_2,h_3\rangle$. This bracket is computed in (\cite{Koch96}, Page 251) as $2h_2h_4$. 

Therefore the sum of the two maps is $4h_2h_4$. It follows that the element $4h_2h_4$ maps to $0$ in $\pi_0 D(\C P^9/ \C P^2)$. This completes the proof.  
\end{proof}

It follows from the above result that the kernel $\pi_{18}^s \to \{\C P^9, S^0\}$ is at least $\Z/2$ and is a subgroup of $\Z/8$. We prove that it cannot be $\Z/8$ so that the kernel can be $\Z/2$ or $\Z/4$. 

\begin{theorem}\label{nontriv18}
The class $h_2h_4\in \pi_{18}^s$ maps to a non-trivial class in $\{\C P^9, S^0\}$.
\end{theorem}

\begin{proof}
We use some computations from \cite{Tod}. Recall that the element $h_2h_4$ is denoted by $\nu^\ast$. For such a stable class the sphere of origin is the first sphere where this class desuspends to. For the class $\nu^*$ the sphere of origin is $S^{12}$ and it desuspends to the class $\xi_{12} : S^{30}\to S^{12}$. The class $\nu^*$ also desuspends to $\nu^*_{16}: S^{34}\to S^{16}$. 

From Lemma 12.14 of \cite{Tod} we know that $H(\nu_{16}^*) = \nu_{31} (\mathit{mod}~2\nu_{31})$. The latter maps isomorphically to the stable range and is equivalent to the map (an odd multiple of) $h_2 \in \pi_3^s$. We have the commutative diagram 
\begin{equation}\label{diag16}
\xymatrix{ [S^{34}, S^{16}] \ar[r]^{H_\#} \ar[d]   & [S^{34}, S^{31}] \ar[d] \ar[r]^{\cong} & \{S^{18}, S^{15}\} \ar[d]  \\
[\Sigma^{16} \C P^9, S^{16}]\ar[r]^{H_\#}     & [\Sigma^{16} \C P^9 , S^{31}] \ar[r]^{\cong} & \{ \C P^9, S^{15}\} } 
\end{equation} 
In terms of the Spanier-Whitehead duality the last map is induced by $\pi_{-15}S^{-18} \to \pi_{-15} D\C P^9$ from the inclusion of the bottom cell. Note that the map $D(\C P^9/\C P^6) \to D\C P^9$ is an isomorphism on $\pi_{-15}$. From the cofibre sequence \ref{cof6} we have the exact sequence 
$$ \cdots \pi_{-15}S^{-15} \to \pi_{-15} S^{-18} \oplus \pi_{-15}S^{-16} \to \pi_{-15} D(\C P^9/\C P^6)\cdots $$
The left map sends $1\in \Z$ to $(h_2,h_1) \in \pi_{-15} S^{-18} \oplus \pi_{-15}S^{-16}$. It follows that $2h_2$ maps to $0$ in $\pi_{-15} D(\C P^9/\C P^6)$ and the class $h_2$ maps non-trivially. Hence in the diagram \ref{diag16} the element $\nu_{16}^*$ in the top left corner maps to a non-trivial element in the bottom right group $\{ \C P^9, S^{15}\}$. Therefore it maps non-trivially in $[\Sigma^{16} \C P^9, S^{16}]$. 

Next we show that this implies that the image in $\{ \C P^9, S^0\}$ is non-trivial. We have the diagram 
$$\xymatrix{
[\Sigma^{16} \C P^9/ \C P^7, S^{16}]  \ar[r] \ar[d]     &   \{ \C P^9/\C P^7, S^0\} \ar[d] \\
[\Sigma^{16} \C P^9, S^{16}]               \ar[r] \ar[d]     &    \{ \C P^9, S^0\}   \ar[d]         \\
[\Sigma^{16} \C P^7, S^{16}]            \ar[r]^{\cong}   &   \{ \C P^9, S^0\}         }$$

Note that the vertical arrows are exact. It follows that the kernel $[\Sigma^{16} \C P^9, S^{16}]      \to  \{ \C P^9, S^0\}$ is isomorphic to the kernel $[\Sigma^{16} \C P^9/ \C P^7, S^{16}]\to \{ \C P^9/\C P^7, S^0\}$. We have observed in the proof of Theorem \ref{inj18} that $\C P^9/\C P^7 \simeq S^{16} \vee S^{18}$. Hence $[\Sigma^{16} \C P^9/ \C P^7, S^{16}]\cong [S^{32},S^{16}] \oplus [S^{34},S^{16}]$ so that the above kernel is direct sum 
$$\mathit{Ker}( [S^{32},S^{16}] \to \{S^{16},S^0\})  \oplus \mathit{Ker}( [S^{34},S^{16}] \to \{S^{18},S^0\})$$
The element $\nu_{16}^\ast \in \pi_{34}S^{16}$ maps to $\nu^* \in \pi_{18}^s$ and so does not lie in the kernel above. It follows that the image of $\nu_{16}^*$ maps non-trivially to $\{\C P^9,S^0\}$. Thus the element $h_2h_4$ is mapped non-trivially in $\{\C P^9,S^0\}$.
\end{proof}

Summarizing the computations in Proposition \ref{inj18}, Theorem \ref{main18}, Theorem \ref{nontriv18}, we have the following result for the Problem \ref{main}.
\begin{cor}\label{ker9}
The kernel $\Ker(f_{\C P^9}^\ast)\subset \pi_{18}^s = \Z/2 \oplus \Z/8$ is non-trivial but not the entire group $\Z/8$. It is either $\Z/2$ or $\Z/4$ as a subgroup of $\Z/8$.  
\end{cor}

\subsection{Computations in higher dimensions}\label{comph}

We follow up the computations in Section \ref{comp18} by demonstrating that there are many examples where the inclusion of the bottom cell in $D(\C P^{8n+2})$ carries a non-trivial kernel in $\pi_0$. The methods here are easier involving computations of Steenrod operations and the existence of certain stable homotopy classes. 

The next example after $18$ is $26$. Recall from \cite{Rav} that the group $\pi_{26}^s\cong \Z/2\{\mu_{26}\}\oplus \Z/2\{h_2^2g\}\oplus \Z/3\{\beta_2\}$.  We know that the class $\mu_{26}$ survives to $\C P^{13}$ from \cite{Far-Jon}. We have the following result

\begin{theorem}\label{main26}
The class $h_2^2g$ maps to $0$ in $\pi_0D(\C P^{13}/\C P^{11})$. It follows that the kernel of the map $\pi_{26}^s \to \pi_0(D\C P^{13})$ is at least $\Z/2$. 
\end{theorem}

\begin{proof}
As in Section \ref{comp18}, we have the filtration 
$$
S^{-26}= D(\C P^{13}/\C P^{12}) \to D(\C P^{13}/\C P^{11}) \cdots D(\C P^{13}/\C P^0)= D(\C P^{13})
$$
and the map $S^{-26} \to D(\C P^{13})$ is the inclusion of the bottom cell. We show that $h_2^2g$ is in the kernel. It suffices to work $2$-locally. We have the cofibre 
$$S^{-25} \to S^{-26} \to D(\C P^{13}/\C P^{11})  $$
The map $S^{-25} \to S^{-26}$ is an element of $\pi_1^s = \Z/2\{h_1\}$ and since $Sq^2(x^{12})=0$ in the cohomology of $\C P^n$,  the map is trivial. It follows that $D(\C P^{13}/\C P^{11}) \simeq S^{-26}\vee S^{-24}$. Next we have the cofibre 
$$S^{-23} \to D(\C P^{13}/ \C P^{11}) \to D(\C P^{13}/\C P^{10})  $$
The map $S^{-23} \to D(\C P^{13}/\C P^{11})$ is given by a pair of maps to $ S^{-26}$ and $S^{-24}$. We note the Steenrod squares $Sq^2(x^{11})=x^{12}$ and $Sq^4(x^{11})=x^{13}$ in the cohomology of $\C P^n$. Therefore the map is given by $(h_1,h_2)$. On $\pi_0$ we have the sequence 
$$\pi_{23}^s \stackrel{(h_1,h_2)}{ \to} \pi_{24}^s \oplus \pi_{26}^s \to \pi_0 D(\C P^{13}/\C P^{10})$$ 
Now note that $h_2g$ represents a non-trivial element in $\pi_{23}^s$ and $h_1h_2g=0$. Thus $h_2^2g$ is in the image of the left hand map and hence maps to $0$ in $ \pi_0 D(\C P^{13}/ \C P^{10})$. Therefore $h_2^2g$ maps to $0$ in $\pi_0D\C P^{13}$ proving the theorem. 
\end{proof}

We have the corresponding result for the Problem \ref{main}.
\begin{cor}\label{ker26}
The kernel $\Ker(f_{\C P^{13}}^\ast)\subset \pi_{26}^s = \Z/2 \oplus \Z/2 \oplus \Z/3$ is non-trivial and contains a summand $\Z/2$.
\end{cor}

Next we show that techniques as in Theorem \ref{main26} exist in many high dimensions. More precisely, we demonstrate examples in higher dimensions where the map $S^{-8n-2} \to D(\C P^{8n+2})$ has a non-trivial kernel on $\pi_0$ using some $p$-local computations. We use the result from \cite{CNL96} :  For $p\geq 7$ the classes $\alpha_1\beta_1^r\gamma_t$ are non trivial in the stable homotopy groups of $S^0$ for $2\leq t \leq p-1$ and $r\leq p-2$ (in dimension $n(t,p,r)= [2(tp^3 - t - p^2) +2r(p^2 - 1 - p) -2]$).  With these assumptions $\beta_1^r\gamma_t$ is also non-trivial in dimension $n(t,p,r)-(2p -3)$. 

We note the proposition 

\begin{prop}
In $H^*(\C P^n; \Z/p)$ $P^1(x^k) \neq 0$ if and only if $p$ does not divide $k$. 
\end{prop}

In the next theorem note that if $(p-1)(t-r) + r \equiv 3~(\mathit{mod}~4)$, $n(t,p,r) \equiv 2(\mathit{mod}~8)$. 
\begin{theorem}\label{mainhigh}
Suppose that $p$ is a prime $\geq 7$, $2\leq t \leq p-1$ and $r\leq p-2$. Assume that $p$ does not divide $t+r$. Then the map $\pi_{n(t,p,r)}^s \to \pi_0 D(\C P^\frac{n(t,p,r)}{2})$ has non-trivial $p$-torsion in the kernel. 
\end{theorem}

\begin{proof}
Note that the condition $p$ does not divide $t+r$ implies that $n(t,p,r)- 2(p-1)$ is not divisible by $p$. We work $p$-locally. The first non-trivial element in $\pi_*^s \otimes \Z_{(p)}$ in positive dimension is $\alpha_1$ in dimension $2p-3$ and this is detected by the Steenrod operation $P^1$. 

Constructing $D(\C P^\frac{n(t,p,r)}{2})$ cell by cell as above in the $p$-local category, the first possible non-trivial attaching map is the map $S^{-n(t,p,r) +2p-3} \to  D(\C P^\frac{n(t,p,r)}{2}/\C P^\frac{n(t,p,r) - 2(p-1)}{2})$. The assumption $p$ does not divide $t+r$ implies that in $\C P^N$ for $N\gg 0$
$$P^1(x^\frac{n(t,p,r)-2(p-1)}{2}) = x^\frac{n(t,p,r)}{2}$$
Therefore the map $S^{-n(t,p,r) +2p-3} \to  D(\C P^\frac{n(t,p,r)}{2}/\C P^\frac{n(t,p,r) - 2(p-1)}{2})$ is given by $\alpha_1$ on the bottom cell. Now in the long exact sequence of homotopy groups
$$\cdots \to \pi_{n(t,p,r)-(2p-3)}^s \to \pi_0 D(\C P^\frac{n(t,p,r)}{2}/\C P^\frac{n(t,p,r) - 2(p-1)}{2}) \to \pi_0 D(\C P^\frac{n(t,p,r)}{2}/\C P^\frac{n(t,p,r) - 2(p-1)}{2})\to \cdots $$
the non-trivial element $\beta_1^r\gamma_t$ maps to the non-trivial element $\alpha_1 \beta_1^r\gamma_t$.  It follows that the non-trivial element $\alpha_1\beta_1^r\gamma_t \in \pi_{n(t,p,r)}^s $ goes to $0$ in $\pi_0 D(\C P^\frac{n(t,p,r)}{2})$. This completes the proof. 
\end{proof}

\begin{remark}
Note that the above conditions are easily satisfied. For example, if $t=3,r=1$, $p\geq 7$ we have $n(t,p,r)= 6p^3 - 2p -10$ which is $\equiv 2 (\mathit{mod}~8)$.
\end{remark}
The following result is immediate consequence of stringing together Theorem \ref{first}, Theorem \ref{coniner}, Corollary \ref{ker9}, Corollary \ref{ker26} and Theorem \ref{mainhigh}.

\begin{theorem}\label{inersumm}

(i) $I(\mathbb{C}P^{9})=\mathbb{Z}_2$ or $\mathbb{Z}_4$.\\
(ii) $I(\mathbb{C}P^{13})\supseteq \Z/2$.\\
(iii) Suppose that $p$ is a prime $\geq 7$, $2\leq t \leq p-1$ and $r\leq p-2$. Assume that $p$ does not divide $t+r$ and $(p-1)(t-r) + r \equiv 3~(\mathit{mod}~4)$ . Then $I(\C P^\frac{n(t,p,r)}{2})\supseteq \Z/p$.

\end{theorem}

\section{Geometric Structures of Inequivalent Smooth Structures}
In this section we explore some consequences of the computations in the previous section. It is known that inequivalent smoothings need not share certain basic geometric
properties with the standard smooth structure. Hitchin has shown in \cite{Hit74} that certain homotopy spheres do not admit a Riemannian metric of positive scalar curvature, while the round metric on the standard $S^m$ has positive sectional curvature. In this section, we also show that there exist two smooth structures on $\mathbb{C} P^{9}$ such that one admits a metric of nonnegative scalar curvature and the other does not. 

 Let $\Ker(d_{\mathbb{R}})$ denotes the kernel of the Adams $d$-invariant $d_{\mathbb{R}} : \pi_{8n+2}^{s}\to \mathbb{Z}/2$ (\cite{Ada}). Under the isomorphism $\Theta_{8n+2}\cong \pi_{8n+2}^{s}$, the Adams $d$-invariant $d_{\mathbb{R}} : \pi_{8n+2}^{s}\to \mathbb{Z}/2$ may be identified with the $\alpha$-invariant homomorphism $\alpha:\Theta_{8n+2}\to \mathbb{Z}/2$ (\cite{Hit74}).  Therefore $\rm{Ker}(d_{\mathbb{R}})$ consists of homotopy $(8n+2)$-sphere which bound spin manifolds.

 In \cite{Ram14}, we studied the Adams $d$-invariant and asked the following question to determine the inertia group $I(\mathbb{C} P^{4n+1})$.
\begin{question}\label{raquestion}
 Let $f:\mathbb{C} P^{4n+1}\to S^{8n+2}$ be any degree one map $(n\geq 1)$.\\
 Does there exist an element $\eta\in \Ker(d_{\mathbb{R}})\subset \pi^s_{8n+2}=\Theta_{8n+2}$ such that the following is true :
 \begin{itemize}
\item [$(\star)$] If any map $h:S^{q+8n+2}\to S^{q}$ represents $\eta$, then $$h\circ \Sigma^{q}f : \Sigma^{q}\mathbb{C} P^{4n+1}\to S^{q}$$ is not null homotopic.
 \end{itemize}
\end{question}
The following theorem shows that the answer to the above question is yes for $n=2$, where the Adams $d$-invariant $d_{\mathbb{R}} : \pi_{18}^{s}=\Z/2\{h_1P^2h_1\} \oplus \Z/8\{h_2h_4\}\to \mathbb{Z}/2$ such that $\Ker(d_{\mathbb{R}})=\Z_8\{h_2h_4\}$ (\cite{Ada}). From Theorem \ref{inersumm} we readily deduce
\begin{prop}
$I(\mathbb{C} P^{9})$ is properly contained in $\Ker(d_{\mathbb{R}})=\mathbb{Z}/8$.
\end{prop}

Picking up an element in $\Ker(d_\R)$ which is not in $I(\C P^9)$, we have a class $\{S^{18},S^0\}$ which maps non-trivially in $\{ \C P^9, S^0\}$. Thus for every representative $S^{q+18}\to S^q$ the corresponding map $\Sigma^q \C P^9 \to S^q$ is not nullhomotopic. This answers the Question \ref{raquestion} in the case $n=2$. 

The existence of classes outside the inertia group for $\C P^9$ also has the following consequence
\begin{theorem}\label{exocom}
There exist three homotopy $18$-spheres $\Sigma^{18}_1$, $\Sigma^{18}_2$ and $\Sigma^{18}_3$ such that the following is true.
\begin{itemize}
\item[\rm{(i)}] The manifolds $\mathbb{C}P^{9}$, $\mathbb{C}P^{9}\#\Sigma^{18}_1$, $\mathbb{C} P^{9}\#\Sigma^{18}_2$ and $\mathbb{C}P^{9}\#\Sigma^{18}_3$ are pairwise non-diffeomorphic.
\item[\rm{(ii)}] The manifolds $\mathbb{C}P^{9}\#\Sigma^{18}_2$ and $\mathbb{C} P^{9}\#\Sigma^{18}_3$ do not admit a metric of nonnegative scalar curvature but $\mathbb{C} P^{9}\#\Sigma^{18}_1$ does.
\end{itemize}
\end{theorem}

\begin{proof}
We start with the first statement. Let $\Sigma^{18}_{1}$, $\Sigma^{18}_2$ and $\Sigma^{18}_3$ be homotopy spheres in $\Theta_{18}\cong \pi_{18}^s=\Z_2\{h_1P^2h_1\} \oplus \Z_8\{h_2h_4\}$ represented by the classes $h_2h_4$, $h_1P^2h_1$ and $h_1P^2h_1+h_2h_4$ respectively. Note that $\Sigma^{18}_{i}\notin I(\mathbb{C} P^{9})$ by Theorem \ref{nontriv18}, where $i=1$, $2$ or $3$, and $\Sigma^{18}_1\in \Ker(d_{\mathbb{R}})$ but $\Sigma^{18}_2$, $\Sigma^{18}_3\notin \Ker(d_{\mathbb{R}})$. If $\mathbb{C} P^{9}\#\Sigma^{18}_{i}$ is diffeomorphic to $\mathbb{C} P^{9}\#\Sigma^{18}_{j}$, then $\Sigma^{18}_{i}\#(\Sigma^{18}_{j})^{-1}\in I(\mathbb{C}P^{9})$. But, $\Sigma^{18}_{i}\#(\Sigma^{18}_{j})^{-1}=\Sigma^{18}_{k}~ or~ (\Sigma^{18}_{k})^{-1}$, where $k\neq i$, $j$. This implies that the manifolds $\mathbb{C} P^{9}$, $\mathbb{C} P^{9}\#\Sigma^{18}_{1}$, $\mathbb{C} P^{9}\#\Sigma^{18}_{2}$ and $\mathbb{C} P^{9}\#\Sigma^{18}_{3}$ are pairwise non-diffeomorphic. This proves (i).

Now we turn to (ii). Since $\Sigma^{18}_i\notin \rm{Ker}(d_{\mathbb{R}})$, where $i=2$ or $3$, and $\mathbb{C}P^{9}$ is a spin manifold equipped with its natural metric (the
Fubini study metric) of positive scalar curvature, the $\alpha$-invariant $\alpha(\mathbb{C}P^{9})=0$ and $\alpha(\Sigma^{18}_i)\neq 0$ (\cite{Hit74}).
 Therefore the $\alpha$-invariant $$\alpha(\mathbb{C}P^{9}\#\Sigma^{18}_i)=\alpha(\mathbb{C}P^{9})+\alpha(\Sigma^{18}_i)\neq 0.$$
 We now proceed by contradiction. Suppose there exists a nonnegative scalar curvature Riemannian metric $g$ on $\mathbb{C}P^{9}\#\Sigma^{18}_i$. Since the non-vanishing of the $\alpha$-invariant and the well-known deformation properties of scalar curvature (\cite{KW75}) now imply that the Riemannian manifold $(\mathbb{C}P^{9}\#\Sigma^{18}_i, g)$ must be scalar-flat and $\mathbb{C}P^{9}\#\Sigma^{18}_i$ has a non-trivial parallel spinor. Therefore the Riemannian manifold $(\mathbb{C}P^{9}\#\Sigma^{18}_i, g)$ has special holonomy (\cite{Hit74}). This is a contradiction, since $\mathbb{C}P^{9}\#\Sigma^{18}_i$ have generic holonomy. Hence $\mathbb{C}P^{9}\#\Sigma^{18}_2$ and $\mathbb{C}P^{9}\#\Sigma^{18}_3$ do not admit a metric of nonnegative scalar curvature. Now consider the connected sum $\mathbb{C}P^{9}\#\Sigma^{18}_1$ and by using the fact $\alpha(\Sigma^{18}_1)=0$, it follows that $\mathbb{C}P^{9}\#\Sigma^{18}_1$ admits a metric of positive scalar curvature (\cite{GL80}). This proves (ii).
\end{proof}

Another application is the following result which is an immediate  consequence of Theorem \ref{exocom} and Theorem \cite[Theorem 1.4]{Ram14}.
\begin{theorem} \label{exocomhyp}
Let $\Sigma^{18}_1$, $\Sigma^{18}_2$ and $\Sigma^{18}_3$ be the specific homotopy $18$-spheres  posited in Theorem \ref{exocom}.
Given a positive real number $\epsilon$, there exists closed complex hyperbolic manifold $M^{18}$ of complex dimension $9$  such that the following is true.
\begin{enumerate}
\item[\rm{(i)}] The manifolds $M^{18}$, $M^{18}\#\Sigma^{18}_1$, $M^{18}\#\Sigma^{18}_2$ and $M^{18}\#\Sigma^{18}_3$ are pairwise non-diffeomorphic.
\item[\rm{(ii)}] Each of the manifolds $M^{18}\#\Sigma^{18}_1$, $M^{18}\#\Sigma^{18}_2$ and $M^{18}\#\Sigma^{18}_3$ supports a negatively curved Riemannian metric whose sectional curvatures all lie in the closed interval $[-4 -\epsilon, -1 +\epsilon]$.
\end{enumerate}
\end{theorem}

\mbox{ }\\

\mbox{ }\\

\end{document}